\documentclass[letterpaper,12pt]{amsart}
\usepackage{latexsym,array,delarray,amsmath,amsthm,amssymb,epsfig,setspace,tikz,ulem,enumitem}

\usepackage{cite}
\usepackage{float}
\usepackage{nicematrix}	
\usepackage{import}
\usepackage[capitalise,noabbrev]{cleveref}

\usepackage{comment}

\usepackage{color}

\theoremstyle{plain}
\newtheorem{thm}{Theorem}[section]

\newtheorem{proposition}[thm]{Proposition}

\newtheorem*{thm*}{Theorem}
\newtheorem*{lemma*}{Lemma}
\newtheorem*{prop*}{Proposition}
\newtheorem*{cor*}{Corollary}
\newtheorem*{conj*}{Conjecture}

\theoremstyle{definition}
\newtheorem{defn}[thm]{Definition}
\newtheorem*{defn*}{Definition}

\newtheorem{ex}[thm]{Example}

\newtheorem{rmk}[thm]{Remark}

\newcommand{\cm}{\mathcal{M}}

\newcommand{\ind}{\mbox{$\perp \kern-5.5pt \perp$}}

\newcommand{\cash}[1]{{\color{blue}(\textbf{\$:} #1)}} %
\newcommand{\john}[1]{{\color{red}(\textbf{J:} #1)}} %
\newcommand{\dev}[1]{{\color{orange}(\textbf{D:} #1)}} %
\newcommand{\zaia}[1]{{\color{purple}(\textbf{Z:} #1)}} %

\title{Graph-based proofs of indistinguishability of linear compartmental models}

\author[Bortner]{Cashous Bortner}
\address{California State University, Stanislaus}
\author[Gilliana]{John Gilliana}
\address{California State University, Stanislaus}
\author[Patel]{Dev Patel}
\address{California State University, Stanislaus}
\author[Tamras]{Zaia Tamras}
\address{California State University, Stanislaus}

\date{\today}

\begin{document}
\maketitle

\begin{abstract}
Given experimental data, one of the main objectives of biological modeling is to construct a model which best represents the real world phenomena.  In some cases, there could be multiple distinct models exhibiting the exact same dynamics, meaning from the modeling perspective it would be impossible to distinguish which model is ``correct.''  This is the study of \textit{indistinguishability} of models, and in our case we focus on linear compartmental models which are often used to model pharmacokinetics, cell biology, ecology, and related fields. Specifically, we focus on a family of linear compartmental models called skeletal path models which have an underlying directed path, and have recently been shown to have the first recorded sufficient conditions for indistinguishability based on underlying graph structure.  In this recent work, certain families of skeletal path models were proven to be indistinguishable, however the proofs relied heavily on linear algebra.  In this work, we reprove several of these indistinguishability results instead using a graph theoretic framework.
\end{abstract}

\normalem

\section{Background} \label{sec:intro}

\subsection{Graph Theory}

\begin{defn}
    A \textit{simple undirected graph} is defined as an ordered pair $(V,E)$, where $V$ (or $V(G)$) is a finite set called \textit{vertices}, and $E$ (or $E(G)$) is a subset of all 2-element subsets of $V$ called \textit{edges}.      A \textit{simple directed graph} is a pair $(V,E)$, where $V$ is a finite set, and $E$ is a subset of $V \times V$.

\end{defn}

\begin{ex}\label{ex:graphs}
The {undirected} graph $G=\left(\{1,2,3,4\},\left\lbrace\{1,2\},\{2,3\},\{3,4\} \right\rbrace \right)$, as seen in \cref{fig:graphs}, has the set of vertices $V(G)=\{1,2,3,4\}$ and the set of edges $E(G) = \left\lbrace\{1,2\},\{2,3\},\{3,4\} \right\rbrace$.  On the other hand, $G' = \left(\{1,2,3,4\},\left\lbrace(1,2),(2,3),(3,4) \right\rbrace \right)$ is the directed graph where the set of vertices $V(G)=\{1,2,3,4\}$ and the set of {directed} edges $E(G) = \left\lbrace(1,2),(2,3),(3,4) \right\rbrace$ represented by arrows as seen in \cref{fig:graphs}.  Note too that we can add {edge labels} (or weights) $a_{21},a_{32},$ and $a_{43}$ corresponding to each edge to $G'$.
\end{ex}

\begin{figure}
        \centering
        \begin{tikzpicture}[scale=1]
 	\draw (0,0) circle (0.3);
 	\draw (2,0) circle (0.3);
 	\draw (4,0) circle (0.3);
        \draw (6,0) circle (0.3);
    	\node[] at (0, 0) {1};
    	\node[] at (2, 0) {2};
    	\node[] at (4, 0) {$3$};
            \node[] at (6,0) {$4$};
	 \draw[-,thick] (0.35, 0) -- (1.65, 0);
	 \draw[-,thick] (2.35,0) -- (3.65,0);
      \draw[-,thick] (4.35,0) -- (5.65,0);
\draw (-1.2,-2) rectangle (7, 2);
    	\node[] at (3, -1.5) {$G$};

     \end{tikzpicture}\begin{tikzpicture}[scale=1]
 	\draw (0,0) circle (0.3);
 	\draw (2,0) circle (0.3);
 	\draw (4,0) circle (0.3);
        \draw (6,0) circle (0.3);
    	\node[] at (0, 0) {1};
    	\node[] at (2, 0) {2};
    	\node[] at (4, 0) {$3$};
            \node[] at (6,0) {$4$};
	 \draw[->,thick] (0.35, 0) -- (1.65, 0);
	 \draw[->,thick] (2.35,0) -- (3.65,0);
      \draw[->,thick] (4.35,0) -- (5.65,0);
   	 \node[] at (1, 0.3) {$a_{21}$};
	\node[] at (3,0.3) {$a_{32}$};
        \node[] at (5,0.3) {$a_{43}$};
\draw (-1.2,-2) rectangle (7, 2);
    	\node[] at (3, -1.5) {$G'$};
    	

\end{tikzpicture}
\caption{On the left is an \textit{undirected} graph $G$ and on the right is a \textit{directed} graph $G'$, both defined in \cref{ex:graphs}.} 
\label{fig:graphs}
\end{figure}
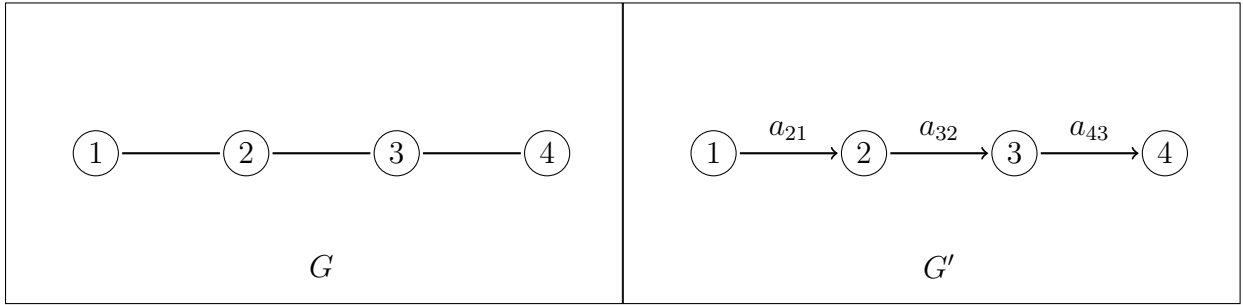

There are many different families of graphs of interest to mathematicians and modelers alike, including cycles, forests, and incoming forests.

\begin{defn}
    A \textit{cycle} is a sequence of edges in an undirected graph that starts and ends on the same vertex without repeating edges.  We can characterize an undirected graph as a \textit{forest} if it contains no cycles.  An \textit{incoming forest} is a directed graph that satisfies two conditions, namely that the underlying, undirected graph (generated by removing the direction of the directed edges) has no cycles, and no vertex in the graph has more than one outgoing edge. 
 \end{defn}

From this definition of incoming forest, we can define the set of all incoming forests within a graph $G$ with $k$ edges to be $\mathcal{F}_k(G)$.  We can also define the incoming forests on a graph $G$ with $k$ edges that also include a path from vertex $i$ to vertex $j$ as $\mathcal{F}_k^{i,j}(G)$.

\begin{ex} 
    The graph $G$ from \cref{ex:graphs} seen in \cref{fig:graphs} is a forest since it has no cycles.  Also, $G'$ from \cref{ex:graphs} is an incoming forest since its underlying undirected graph is $G$ (no cycles), and no vertex in the graph has more than one outgoing edge.  Specifically, each of the first three vertices has exactly one outgoing edge, while the fourth vertex has zero outgoing edges.

    We can also see that the set of incoming forests with two edges in $G'$ is $$\mathcal{F}_2(G') = \left\lbrace \{1 \to 2, 2 \to 3\} , \ \{1 \to 2, 3 \to 4\}, \ \{2 \to 3, 3 \to 4\} \right\rbrace. $$
    The set of incoming forests with two edges in $G'$ and a path from $1$ to $2$ is
    \[
    \mathcal{F}_2^{1,2}(G') = \left\lbrace \{1 \to 2, 2 \to 3\} , \ \{1 \to 2, 3 \to 4\} \right\rbrace.
    \]
\end{ex}

\begin{defn}
    The \textit{productivity} of a graph $G$ with edge set $E(G)$ is the product of its edge labels:
\begin{align} \label{eq:productivity}
	\pi_G ~:=~ \prod_{ e \in E(G) } L(e)~,
\end{align}
where $L(e)$ is the label of edge $e$.
Following the usual convention, we define $\pi_G=1$ for graphs $G$ having no edges.  
\end{defn}

\begin{ex} 
    The productivity of the graph $G'$ from \cref{ex:graphs} seen in \cref{fig:graphs} is 
    \[
    \pi_{G'} = a_{21}a_{32}a_{43}.
    \]
\end{ex}


\subsection{Linear Compartmental Models}

\begin{defn}
    A \textit{linear compartmental model} (LCM) is a model $\cm=(G,In,Out,Leak)$ where $G = (V, E)$ is a directed graph consisting of vertices $V$ and edges $E$, with input and output vertices along with leaks $In, Out, Leak\subseteq V$.
\end{defn}

There are several families of linear compartmental models of interests to both modelers and biologists.  Linear compartmental models have been used to study pharmacokinetics, ecology, and epidemology \cite{godfrey,dipiro2010concepts,hedaya2012basic,tozer1981concepts,wagner1981history,blackwood2018introduction,gydesen1984mathematical}.  Physiological models could involve metabolism, biliary, or excretory pathways \cite{distefano-book}.  Previous work has been done considering certain families of linear compartmental models, including cycle, mammillary, catenary, and tree models \cite{singularlocus,aim,MeshkatSullivantEisenberg,CJSSS}. 

\begin{defn}\label{def:skeletalpathmodels}  A model $\cm = (G, In, Out, Leak)$ is a \textit{skeletal path model} if $G$ contains a directed path from compartment $1$ to $n$, with $In=\{1\}$ and $Out=\{n\}$.  In other words, skeletal path models will have a graph that contains a ``backbone'' of a path along with other vertices, edges, or leaks.
\end{defn}

\begin{ex}\label{ex:path}
    Consider the skeletal path model $\cm_3= (P_4, \{1\},\{4\},\{3\})$ seen in Figure \ref{fig:path} where $P_4$ is the directed path from vertex $1$ to vertex $4$ with edges $\{1\to 2, \ 2\to 3,\ 3 \to 4\}$.  We will denote the parameters of this model as $\mathcal{P}(\cm_3)=\{a_{03},a_{21},a_{32},a_{43}\}$.

    \begin{figure}
        \centering
        \begin{tikzpicture}[scale=1]
 	\draw (0,0) circle (0.3);
 	\draw (2,0) circle (0.3);
 	\draw (4,0) circle (0.3);
        \draw (6,0) circle (0.3);
    	\node[] at (0, 0) {1};
    	\node[] at (2, 0) {2};
    	\node[] at (4, 0) {$3$};
            \node[] at (6,0) {$4$};
	 \draw[->] (0.35, .1) -- (1.65, .1);
	 \draw[->] (2.35,.1) -- (3.65,.1);
      \draw[->] (4.35,.1) -- (5.65,.1);
   	 \node[] at (1, 0.3) {$a_{21}$};
	\node[] at (3,0.3) {$a_{32}$};
        \node[] at (5,0.3) {$a_{43}$};
	\draw (6.69,.69) circle (0.07);	
	 \draw[-] (6.65, .65 ) -- (6.22, .22);
	 \draw[->] (-.65, .65) -- (-.25, .25);	
   	 \node[] at (-.8,.8) {in};
	 \draw[->] (4,-.3) -- (4, -.9);	
   	 \node[] at (4.35, -.7) {$a_{03}$};
\draw (-1.4,-2) rectangle (7, 2);
    	\node[] at (3, -1.5) {$\cm_3 = (P_4, \{1\},\{4\},\{3\})$};
    	

\end{tikzpicture}
        \caption{The model described in Example \ref{ex:path}}
        \label{fig:path}
    \end{figure}
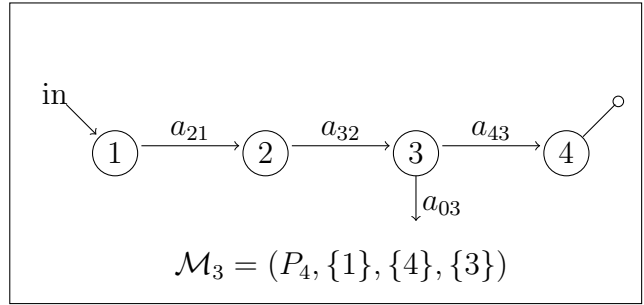
    
\end{ex}

Now that we have shown an example of a particular skeletal path model, we will use it to motivate the definition of \textit{differential equations} regarding LCMs and a \textit{compartmental matrix} of an LCM.

\begin{defn}
    A \textit{compartmental matrix} $A$ corresponding to a linear compartmental model $\cm=(G,In,Out,Leak)$ is given by:
\[
  A_{ij} = \left\{ 
  \begin{array}{l l l}
    -a_{0i}-\sum_{k: i \rightarrow k \in E}{a_{ki}} & \quad \text{if $i=j$ and } i \in Leak\\
        -\sum_{k: i \rightarrow k \in E}{a_{ki}}     & \quad \text{if $i=j$ and } i \notin Leak\\
    a_{ij} & \quad \text{if $j\rightarrow{i}$ is an edge of $G$}\\
    0 & \quad \text{otherwise}\\
  \end{array} \right.
\]
\end{defn}
From this definition of the compartmental matrix we can build the associated linear system of ordinary differential equations as follows: \\
Given the inputs and outputs associated to the quadruple
$(G, In, Out, Leak)$,
\begin{equation*} \label{eq:main}
\dot{x}(t)=Ax(t)+u(t)  \quad \quad y_j(t)=x_j(t)  \mbox{ for } j \in Out
\end{equation*}
 where $u_{i}(t) = 0$ for $i \notin In$.
 The coordinate functions $x_{i}(t)$ are the state variables, the 
 functions $y_{j}(t)$ are the output variables, and the nonzero functions $u_{i}(t)$ are
 the inputs.  The resulting model is called a   \textit{linear compartmental model}. 

Note that the compartmental matrix associated to a linear compartmental model with no leaks has the form of a weighted Laplacian matrix of the underlying graph.

\begin{ex}
    The system of differential equations defining the model $\cm_3$ in \cref{ex:path} is given by

    \begin{align}
\label{eq:iofind1}        x_1' &= -a_{21}x_1+u_1\\
\label{eq:iofind2}        x_2' &= a_{21}x_1 - a_{32}x_2\\
\label{eq:iofind3}        x_3' &= a_{32}x_2-(a_{03}+a_{43})x_3\\
\label{eq:iofind4}        x_4' &= a_{43}x_3 \\
\label{eq:iofind5}        y_4 &= x_4
    \end{align}
    or in matrix form

    \begin{align*}
        \begin{pmatrix}
            x_1' \\ x_2' \\ x_3' \\ x_4'
        \end{pmatrix} &= \begin{pmatrix}
            -a_{21} & 0 & 0 & 0 \\
            a_{21} & -a_{32} & 0 & 0 \\
            0 & a_{32} & -a_{03} - a_{43} & 0 \\
            0 & 0 & a_{43} & 0
        \end{pmatrix}
        \begin{pmatrix}
            x_1 \\ x_2 \\ x_3 \\ x_4
        \end{pmatrix} + \begin{pmatrix}
            u_1 \\ 0 \\0 \\ 0
        \end{pmatrix} \\
        \\
        y_4 &=x_4.
    \end{align*}
\end{ex}

\begin{defn}
    An \textit{input-output equation} of a linear compartmental model $\cm=(G,\{in\}\{out\},Leak\}$ with exactly one input vertex and one output vertex consists of an ordinary differential equation in the input and output variables $u_{in}$ and $y_{out}$ with coefficients $\lbrace c_0, c_1, \dots, c_n, d_0,d_1,\ldots , d_{n-1}\rbrace$ that are functions of the parameter edge weights: \\
    \\
     $ y_{out}^{(n)} + c_{n-1}  y_{out}^{(n-1)} + \dots  + c_1 y_{out}' + c_0 y_{out} ~=~ 
		d_{n-1} u_{in}^{(n-1)} + d_{n-2} u_{in}^{(n-2)} +
			 \dots  + d_{1} u_{in}' + d_{0} u_{in}.$ \\

\end{defn}

\begin{ex} \label{ex:ioequation}
One method of finding an input-output equation associated to the model $\cm_3$ in \cref{ex:path} is to use differential elimination and substitution to eliminate all of the state variables $x_i$.  First, we take the fact that $y_4=x_4$ in \cref{eq:iofind5}, take the derivative of this equality, $y_4'=x_4'$, and substitute in \cref{eq:iofind4} yielding
\[
y_4' = a_{43}x_3.
\]
Solving for $x_3$, we get $x_3 = \frac{1}{a_{43}}y_4'$, which we can differentially substitute into \cref{eq:iofind3}, to get
\begin{align*}
\frac{1}{a_{43}}y_4'' &= a_{32}x_2 -(a_{03}+a_{43})\frac{1}{a_{43}}y_4' \\
&\frac{1}{a_{43}}y_4'' = a_{32}x_2 - \left( \frac{a_{03}}{a_{43}} + 1 \right) y_4'.
\end{align*}

Now solving for $x_2$, we get
\[
x_2 = \frac{1}{a_{43}a_{32}}y_4'' + \left( \frac{a_{03}}{a_{43}a_{32}} + \frac{1}{a_{32}} \right) y_4'
\] \\
and taking the derivative we get
\begin{align*}  
x_2' = \frac{1}{a_{43}a_{32}}y_4''' + \left( \frac{a_{03}}{a_{43}a_{32}} + \frac{1}{a_{32}} \right) y_4''
\end{align*}

Now, differentially substitute into \cref{eq:iofind2} to get
 
 \begin{align*}
     \frac{1}{a_{43}a_{32}}y_4''' + \left( \frac{a_{03}}{a_{43}a_{32}} + \frac{1}{a_{32}} \right) y_4''
 &= a_{21}x_1-a_{32}\left( \frac{1}{a_{43}a_{32}}y_4'' + \left( \frac{a_{03}}{a_{43}a_{32}} + \frac{1}{a_{32}} \right) y_4' \right) \\
 & \\
 &= a_{21}x_1-\frac{a_{32}}{a_{43}}y_4'' -  \left( \frac{a_{03}}{a_{43}} + 1 \right) y_4'
 \end{align*}

Now solving for $x_1$ we have the following :

\[
x_1 = \frac{1}{a_{43}a_{32}}y_4''' + \left( \frac{a_{03}}{a_{43}a_{32}} + \frac{1}{a_{32}} + \frac{a_{32}}{a_{43}} \right) y_4'' + \left( \frac{a_{03}}{a_{43}} + 1 \right) y_4'
\]

Differentiating, we get
\[
x_1' = \frac{1}{a_{43}a_{32}}y_4^{(4)} + \left( \frac{a_{03}}{a_{43}a_{32}} + \frac{1}{a_{32}} + \frac{a_{32}}{a_{43}} \right) y_4''' + \left( \frac{a_{03}}{a_{43}} + 1 \right) y_4''
\]

Now, we can differentially substitute into \cref{eq:iofind1} to get the following:

\begin{align*}
    \frac{1}{a_{43}a_{32}}y_4^{(4)} &+ \left( \frac{a_{03}}{a_{43}a_{32}} + \frac{1}{a_{32}} + \frac{a_{32}}{a_{43}} \right) y_4''' + \left( \frac{a_{03}}{a_{43}} + 1 \right) y_4''  \\ 
    & = -a_{21} \left( \frac{1}{a_{43}a_{32}}y_4''' + \left( \frac{a_{03}}{a_{43}a_{32}} + \frac{1}{a_{32}} + \frac{a_{32}}{a_{43}} \right) y_4'' + \left( \frac{a_{03}}{a_{43}} + 1 \right) y_4' \right) +u_1
\end{align*}

Now we can move our output variables to the left and our input variable to the right, and multiply by a common factor of $a_{32}a_{43}$ to clear the denominators yielding:

\begin{align*}y_{4}^{(4)}+(a_{21}+a_{32}+a_{03}+a_{43})y_{4}^{(3)} &+ (a_{21}a_{32} + a_{21}a_{03} + a_{21}a_{43} +  a_{32}a_{03} + a_{32}a_{43})y_{4}^{(2)} \\ 
&+(a_{21}a_{32}a_{43} + a_{21}a_{32}a_{03})y_{4}^{(1)}  = (a_{21}a_{32}a_{43})u_{1}
\end{align*}

\end{ex}

There are other methods for determining an input-output equation of a linear compartmental model, including a clever use of Cramer's Rule \cite{MeshkatSullivantEisenberg}, or a use of transfer functions \cite{Ovchinnikov-Pogudin-Thompson}.  Unfortunately, these methods do not yield much intuition into the relationship between the input-output equation[s] (and corresponding characteristics), and the underlying graph of the model.  Fortunately, as described in the next section, there is a method for computing an input-output equation of a linear compartmental model directly utilizing the structure of the underlying graph.

From an input-output equation corresponding to a model, we can define a polynomial map from the space of parameter values to the space of coefficients as follows:
\begin{defn}
    The \textit{coefficient map} corresponding to a linear compartmental model is the map from the space of parameters to the space of coefficients of its input-output equation.
\end{defn}

\begin{ex}
 The coefficient map of the model $\cm_3$ seen in \cref{ex:path} with corresponding input-output equation seen in \cref{ex:ioequation} is given by:
    \begin{align*}
    \mathbf{c}: \mathbb{R}^4 &\to \mathbb{R}^4 \\
    \begin{pmatrix}
        a_{03} \\ 
        a_{21} \\
        a_{32} \\
        a_{43}
    \end{pmatrix} &\mapsto 
    \begin{pmatrix}
        a_{21}+a_{32}+a_{03}+a_{43} \\
        a_{21}a_{32} + a_{21}a_{03} + a_{21}a_{43} +  a_{32}a_{03} + a_{32}a_{43} \\
        a_{21}a_{32}a_{43} + a_{21}a_{32}a_{03}\\ 
        a_{21}a_{32}a_{43}
    \end{pmatrix}
    \end{align*}
\end{ex}

There are many interesting questions related to the coefficient map of linear compartmental models.  For example, modelers are often interested in whether they can recover (or identify) the parameter values of a model while only knowing information about the input and output equation.  One approach to this \textit{identifiability} problem is to consider the injectivity of this coefficient map. \cite{aim,MeshkatSullivantEisenberg,meshkat-rosen-sullivant,CJSSS,bortner-meshkat,Ovchinnikov-Pogudin-Thompson}.

We consider another characteristic of linear compartmental models which we define in the next section.

\subsection{Indistinguishability}\label{sect:indist}

\begin{defn}
     We say two models $\cm$ and $\cm'$ are \textit{indistinguishable via a permutation of parameters} or \textit{permutation indistinguishable} if models $\cm$ and $\cm'$ have the same input-output equations up to a renaming of the parameters.  This means there is a bijection $\Phi$ from the parameters in $\cm$ to the parameters in $\cm'$ such that the coefficients of $\cm'$ are exactly the coefficients of $\cm$ under $\Phi$.
\end{defn}

There are other, more general notions of indistinguishability as described in \cite{BortnerMeshkat-Indist} and \cite{godfrey-chapman}, however in this work we will focus exclusively on permutation indistinguishability.  Though we will not formally define general indistinguishability, as discussed by Bortner and Meshkat \cite[Proposition 3.13]{BortnerMeshkat-Indist} knowing that if two models are permutation indistinguishable, then they are indistinguishable, though the converse of this statement is not true.

\begin{ex}\label{ex:indist}
 \begin{figure}
        \centering
        \begin{tikzpicture}[scale=1]
 	\draw (0,0) circle (0.3);
 	\draw (2,0) circle (0.3);
 	\draw (4,0) circle (0.3);
        \draw (6,0) circle (0.3);
    	\node[] at (0, 0) {1};
    	\node[] at (2, 0) {2};
    	\node[] at (4, 0) {$3$};
            \node[] at (6,0) {$4$};
	 \draw[->] (0.35, .1) -- (1.65, .1);
	 \draw[->] (2.35,.1) -- (3.65,.1);
      \draw[->] (4.35,.1) -- (5.65,.1);
      \draw[<-] (4.35,-.12) -- (5.65,-.12);
   	 \node[] at (1, 0.3) {$a_{21}$};
	\node[] at (3,0.3) {$a_{32}$};
        \node[] at (5,0.3) {$a_{43}$};
        \node[] at (5, -0.3) {$a_{34}$};
	\draw (6.69,.69) circle (0.07);	
	 \draw[-] (6.65, .65 ) -- (6.22, .22);
	 \draw[->] (-.65, .65) -- (-.25, .25);	
   	 \node[] at (-.8,.8) {in};
\draw (-1.4,-2) rectangle (7, 2);
    	\node[] at (3, -1.5) {$\cm_4 = (H, \{1\},\{4\},\emptyset)$};
    	

\end{tikzpicture}
        \caption{The model described in Example \ref{ex:indist} where $H=P_4 \cup \{4 \to 3\}$}
        \label{fig:indist}
    \end{figure}
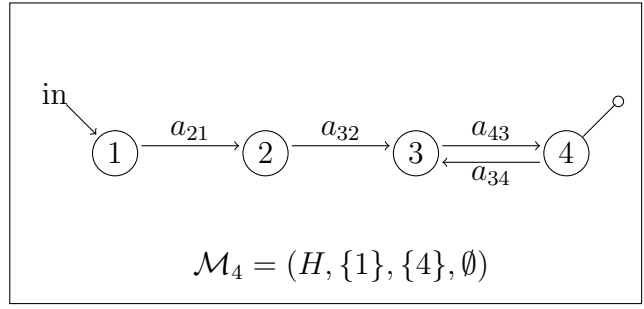
The input-output equation for the model $\cm_4=(H,\{1\},\{4\},\emptyset)$ where $H=P_4 \cup \{4 \to 3\}$ seen in \cref{fig:indist} is given by:

\begin{align*}y_{4}^{(4)}+(a_{21}+a_{32}+{\color{red}a_{34}}+a_{43})y_{4}^{(3)} &+ (a_{21}a_{32} + a_{21}{\color{red}a_{34}} + a_{21}a_{43} +  a_{32}{\color{red}a_{34}} + a_{32}a_{43})y_{4}^{(2)} \\ 
&+(a_{21}a_{32}a_{43} + a_{21}a_{32}{\color{red}a_{34}})y_{4}^{(1)}  = (a_{21}a_{32}a_{43})u_{1}.
\end{align*}
If we compare this to the input-output equation of the model seen in \cref{fig:path} and calculated in \cref{ex:ioequation}:
\begin{align*}y_{4}^{(4)}+(a_{21}+a_{32}+{\color{red}a_{03}}+a_{43})y_{4}^{(3)} &+ (a_{21}a_{32} + a_{21}{\color{red}a_{03}} + a_{21}a_{43} +  a_{32}{\color{red}a_{03}} + a_{32}a_{43})y_{4}^{(2)} \\ 
&+(a_{21}a_{32}a_{43} + a_{21}a_{32}{\color{red}a_{03}})y_{4}^{(1)}  = (a_{21}a_{32}a_{43})u_{1}
\end{align*}
we can see that these two models are permutation indistinguishable via renaming defined by:
\begin{align*}
    \Phi \colon \mathcal{P}(\cm_3) &\to \mathcal{P}(\cm_4) \\
    {\color{red}a_{03}} &\mapsto {\color{red}a_{34}} \\
    a_{21} &\mapsto a_{21} \\
    a_{32} &\mapsto a_{32} \\
    a_{43} &\mapsto a_{43}
\end{align*}

\end{ex}

\subsection{Elementary Symmetric Polynomials}


Now we introduce the notion of elementary symmetric polynomials which will help us in generating input-output equations in our main results.

\begin{defn}
    For $n$ variables $Q = \{x_1, x_2, \ldots , x_n\}$, the $k$-th \textit{elementary symmetric polynomial}, denoted $\sigma_k(Q)$,  is the sum of all distinct products of $k$ variables in $Q$.
\end{defn}

Note that given $n$ variables, the $k$-th elementary symmetric polynomial is a homogeneous polynomial consisting of a sum of $\binom{n}{k}$ monomials each of degree $k$.

\begin{ex}\label{ex:elemsymm}
   Assuming we have $Q=\{x_1,x_2,x_3,x_4\}$, then we can compute the following elementary symmetric polynomials:
   \begin{align*}
   \sigma_0(Q) &= 1 \\ 
   \sigma_1(Q) &= x_1 + x_2 + x_3 + x_4 \\ 
   \sigma_2(Q) &= x_1x_2 + x_1x_3 + x_1x_4 + x_2x_3 + x_2x_4 + x_3x_4 \\ 
   \sigma_3(Q) &= x_1x_2x_3 + x_1x_2x_4 + x_1x_3x_4 + x_2x_3x_4 \\
   \sigma_4(Q) &= x_1x_2x_3x_4
   \end{align*}
\end{ex}

\begin{rmk}
    Moving forward, we assume that $\sigma_0(Q) =1$ and $\sigma_i(Q) = 0$ for $i<0$ or $Q=\emptyset$.

\end{rmk}

\section{Previous Work}

\begin{defn} For a linear compartmental model $\cm=(G,\{in\},\{out\},Leak)$, define the following associated graphs:
\begin{itemize}
    \item $\widetilde{G}$ is the graph $G$ with an added vertex ``0'' such that for every $j\in Leak$, $j\to 0 \in E(\widetilde{G})$, and
    \item $\widetilde{G}^\ast$ is the graph $\widetilde{G}$ such that every edge that leaves the $out$ output node is removed.
\end{itemize}
\end{defn}

\begin{ex}
In \cref{fig:tilde}, on the left we have the $\widetilde{P}_4$ graph associated to the model $\cm_3$ from \cref{ex:path}.  Note that this is also the $\widetilde{P}_4^*$ graph associated to this model since the $\widetilde{P}_4$ graph does not have any edges which leave the output node $4$.

The graph on the right of \cref{fig:tilde} represents the $\widetilde{H}^*$ graph associated to the model $\cm_4$ as defined in \cref{ex:indist}.  Note that the main difference between this graph and that of the original $H$ is that the edge $4\to 3$ must be removed since $4$ is the output vertex.  Also, since $\cm_4$ has no leaks, we ignore the addition of the $0$ vertex.

    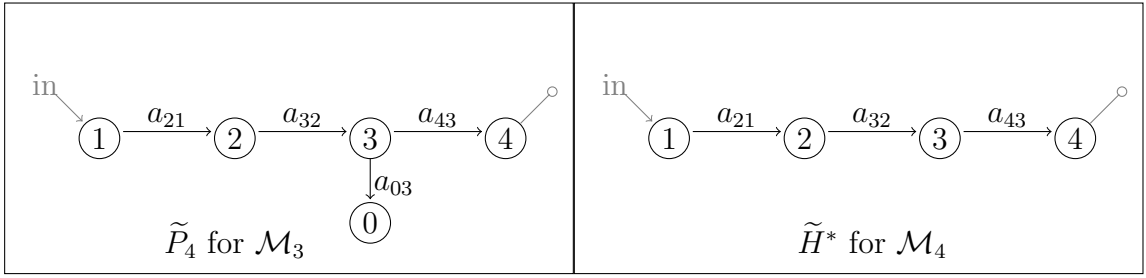
\begin{figure}
        \centering
        \begin{tikzpicture}[scale=.9]
 	\draw (0,0) circle (0.3);
 	\draw (2,0) circle (0.3);
 	\draw (4,0) circle (0.3);
        \draw (6,0) circle (0.3);
    	\node[] at (0, 0) {1};
    	\node[] at (2, 0) {2};
    	\node[] at (4, 0) {$3$};
            \node[] at (6,0) {$4$};
	 \draw[->] (0.35, .1) -- (1.65, .1);
	 \draw[->] (2.35,.1) -- (3.65,.1);
      \draw[->] (4.35,.1) -- (5.65,.1);
   	 \node[] at (1, 0.3) {$a_{21}$};
	\node[] at (3,0.3) {$a_{32}$};
        \node[] at (5,0.3) {$a_{43}$};
	\draw[gray] (6.69,.69) circle (0.07);	
	 \draw[-,gray] (6.65, .65 ) -- (6.22, .22);
	 \draw[->,gray] (-.65, .65) -- (-.25, .25);	
   	 \node[] at (-.8,.8) {\color{gray}in};
	 \draw[->] (4,-.3) -- (4, -.9);	
   	 \node[] at (4.35, -.7) {$a_{03}$};
     \draw (4,-1.25) circle (.3);
        \node[] at (4,-1.25) {0};
\draw (-1.4,-2) rectangle (7, 2);
    	\node[] at (2, -1.5) {$\widetilde{P}_4$ for $\cm_3$};
    	

\end{tikzpicture}\begin{tikzpicture}[scale=.9]
 	\draw (0,0) circle (0.3);
 	\draw (2,0) circle (0.3);
 	\draw (4,0) circle (0.3);
        \draw (6,0) circle (0.3);
    	\node[] at (0, 0) {1};
    	\node[] at (2, 0) {2};
    	\node[] at (4, 0) {$3$};
            \node[] at (6,0) {$4$};
	 \draw[->] (0.35, .1) -- (1.65, .1);
	 \draw[->] (2.35,.1) -- (3.65,.1);
      \draw[->] (4.35,.1) -- (5.65,.1);
   	 \node[] at (1, 0.3) {$a_{21}$};
	\node[] at (3,0.3) {$a_{32}$};
        \node[] at (5,0.3) {$a_{43}$};
	\draw[gray] (6.69,.69) circle (0.07);	
	 \draw[-,gray] (6.65, .65 ) -- (6.22, .22);
	 \draw[->,gray] (-.65, .65) -- (-.25, .25);	
   	 \node[] at (-.8,.8) {\color{gray}in};
\draw (-1.4,-2) rectangle (7, 2);
    	\node[] at (3, -1.5) {$\widetilde{H}^*$ for $\cm_4$};
    	

\end{tikzpicture}

        \caption{On the left is the $\widetilde{P}_4$ graph corresponding to $\cm_3$ as defined in \cref{ex:path}.  On the right is the $\widetilde{H}^*$ graph corresponding to $\cm_4$ as defined in \cref{ex:indist}.}
        \label{fig:tilde}
    \end{figure}

\end{ex}

These graphs associated to the underlying graph of a model $\cm$ can be used to generate the input-output equation of a model using the following result:

\begin{thm}[Theorem 3.1 from \cite{aim}]\label{thm:ioeq}
Consider a linear compartmental model $\mathcal{M}=(G, \{in\}, \{out\},  Leak)$ with a single input and output and let $n=|V(G)|$ denote the number of compartments.  Write the input-output equation as: 
	\begin{equation}\label{eq:inout}
	y_{out}^{(n)}+c_{n-1}y_{out}^{(n-1)}+\cdots + c_1y_{out}'+c_0y_{out}
	 ~= ~ d_{n-1}u_{in}^{(n-1)} + d_{n-2} u_{in}^{(n-2)}+ \cdots +d_1u_{1}'+d_0u_{in}~.
	\end{equation} 
Then the coefficients of this input-output equation are as follows { (where $\pi_F$ is as in~\eqref{eq:productivity})}:
\begin{align*}
c_i &= \sum_{F \in \mathcal{F}_{n-i}(\widetilde{G})} \pi_F \quad \text{ for } i=0,1,\ldots , n-1~, \ \text{ and } \\
d_i &= \sum_{F \in \mathcal{F}_{n-i-1}^{in,out}(\widetilde{G}^\ast)} \pi_F \quad \text{ for } i=0,1,\ldots, n-2~.
\end{align*}
\end{thm}

Note that we have written this theorem for the case of a single input and output, but the original statement of the theorem will give input-output equations for multiple inputs and/or outputs.

\begin{ex}\label{ex:iobygraph} 



First, we can find the incoming forests on $\widetilde{P}_4$ from the model $\cm_3$ in \cref{ex:path} based on the number of edges (and represented by edge weight for ease) as:
\begin{align*}
    \mathcal{F}_1\left(\widetilde{P}_4\right) &= \left\lbrace \{a_{03}\}, \ \{a_{21}\},\ \{a_{32}\},\ \{a_{43}\}
    \right\rbrace \\
        \mathcal{F}_2\left(\widetilde{P}_4\right) &= \left\lbrace \{a_{21},a_{32}\}, \ \{a_{21,}a_{43}\},\  \{a_{21},a_{03}\}, \ \{a_{32},a_{43}\}, \ \{a_{32},a_{03}\}
    \right\rbrace \\
    \mathcal{F}_3\left(\widetilde{P}_4\right) &= \left\lbrace \{a_{21},a_{32},a_{43}\}, \ \{a_{21},a_{32},a_{03}\} \right\rbrace \\
    \mathcal{F}_4\left(\widetilde{P}_4\right) &= \left\lbrace \right\rbrace
\end{align*}
Note that there are no incoming forests with four edges in the graph $\widetilde{P}_4$ since the inclusion of all four edges in $\widetilde{P}_4$ would include two outgoing edges from vertex 3.

We can also find the incoming forests on $\widetilde{P}_4^*$ including a path from the input (1) to the output (4) as seen in \cref{fig:tilde}:
\begin{align*}
\mathcal{F}_1^{1,4}\left(\widetilde{P}_4^*\right) &= \left\lbrace \right\rbrace  \\
\mathcal{F}_2^{1,4}\left(\widetilde{P}_4^*\right) &= \left\lbrace \right\rbrace  \\
\mathcal{F}_3^{1,4}\left(\widetilde{P}_4^*\right) &= \left\lbrace \{a_{21},a_{32},a_{43}\}\right\rbrace 
\end{align*}
Note here that we cannot have an incoming forest include a path from 1 to 4 with less than three edges since the distance from 1 to 4 is three.  Also, for the same reason as in the incoming forests on $\widetilde{P}_4$, we cannot include all four edges because we would have two outgoing edges from vertex 3.

Now, if we again consider the input-output equation of the model $\cm_3$, we see that the coefficients of the left-hand side correspond exactly to the sum of the productivities of each of the sets in $\mathcal{F}_k\left( \widetilde{P}_4 \right)$, and the right-hand side coefficient is exactly the producivity of the single incoming forest on $\widetilde{P}_4^*$ that includes a path from 1 to 4:

\begin{align*}y_{4}^{(4)}+\overbrace{(a_{21}+a_{32}+a_{03}+a_{43})}^{\mathcal{F}_1\left(\widetilde{P}_4\right)}y_{4}^{(3)} &+ \overbrace{(a_{21}a_{32} + a_{21}a_{03} + a_{21}a_{43} +  a_{32}a_{03} + a_{32}a_{43})}^{\mathcal{F}_2\left(\widetilde{P}_4\right)}y_{4}^{(2)} \\ 
&+\underbrace{(a_{21}a_{32}a_{43} + a_{21}a_{32}a_{03})}_{\mathcal{F}_3\left(\widetilde{P}_4\right)}y_{4}^{(1)}  = \underbrace{(a_{21}a_{32}a_{43})}_{\mathcal{F}_3^{1,4}\left(\widetilde{P}_4^*\right)}u_{1}
\end{align*}
Again, because we have no incoming forests with four edges on $\widetilde{P}_4$, the coefficient of $y_4^{(0)}$ is zero.  Likewise, the coefficients of both $u_1^{(1)}$ and $u_1^{(2)}$ are also zero since there are no incoming forests on $\widetilde{P}_4^*$ containing a path from 1 to 4.
\end{ex}

As previously mentioned, the main focus of this work is on the indistinguishability characteristic of linear compartmental models.  This problem first appears in the literature in the early 1990's as Godfrey and Chapman describe four necessary conditions for more general indistinguishability of linear compartmental models \cite{godfrey-chapman}.  Shortly thereafter, Zhang, Collins, and King \cite{zhang1991} developed algorithms to check that these properties described by Chapman and Godfrey held, and to manually check indistinguishability of linear compartmental models.  More recently, work by Bortner and Meshkat \cite{BortnerMeshkat-Indist} have justified these necessary conditions given by Godfrey and Chapman in the context of the graph theoretic approach to generating input-output equations, and also derived the first known sufficient conditions for [permutation] indistinguishability of a family of linear compartmental models.  Specifically, Bortner and Meshkat focus on this class of skeletal path models as was defined in \cref{def:skeletalpathmodels}.  In that work, Bortner and Meshkat develop sufficient conditions for permutation indistinguishability of different types of skeletal path models, including the following two results:

First, every path models consisting of only the path from the input to the output and a single leak somewhere other than the output node are indistinguishable, or more formally:
\begin{thm} 
    The path models $\cm_i=(P_n,\{1\},\{n\},\{i\})$ and $\cm_k = (P_n,\{1\},\{n\},\{k\})$ are [permutation] indistinguishable for all $1\leq i <k <n$.
\end{thm}

Second, the path model with a leak in the second to last compartment is indistinguishable with the path model with no leaks but the additional edge from the last node to the second to last node in the path:
\begin{thm} 
    The path models $\cm_{n-1}=(P_n,\{1\},\{n\},\{n-1\})$ and $\cm_n = (H,\{1\},\{n\},\emptyset)$ where $H=P_n \cup \{n \to n-1\}$ are [permutation] indistinguishable.
\end{thm}

Meshkat and Bortner also discuss the fact that permutation indistinguishability is an equivalence relation, so by transitivity this suggests that a model with a leak in any compartment excluding the output compartment is [permutation] indistinguishable with this $\cm_n$ model.

These results derived by Bortner and Meshkat were proven using linear algebra, with an emphasis on the aforementioned Cramer's rule approach to generating the input-output equations.  The purpose of this work is to reprove these results using the graph theoretic generation of the input-output equation as a proof of concept for future indistinguishability work.

\section{Indistinguishability Proofs via Graph Theory}

Before we prove the indistinguishability results, we first develop a generation of the input-output equation of these specific skeletal path models using the graph theoretic generation described in \cref{thm:ioeq}.

\begin{proposition}\label{prop:firstsigma}
    Consider a skeletal path model $\mathcal{M}_i=(P_n, \{1\}, \{n\},  \{i\})$ with a single input and output, and $i<n$.  Write the input-output equation as: 
	\begin{equation}\label{eq:inout}
	y_{n}^{(n)}+c_{n-1}y_{n}^{(n-1)}+\cdots + c_1y_{n}'
	 ~= ~ d_0u_{1}~.
	\end{equation} 
 Let $Q_i=\mathcal{P}(\cm_i)$.  Then the coefficients of this input-output equation are as follows:
\begin{align*}
c_j &= \sigma_{n-j}(Q_i) - \left(a_{0i}a_{(i+1)i}\right)\cdot  \sigma_{n-j-2}(Q_i\setminus\{a_{0i},a_{i+1,i}\}) \ \text{ for } j\in \{1,2,\ldots, n-1\}, \\
d_0 &= a_{21}a_{32}\cdots a_{n(n-1)}.
\end{align*}    
\end{proposition}

\begin{proof}
Input-output equations of linear compartmental models can be generated using two specific criterion, given \cref{thm:ioeq}, 1) incoming and 2) forest conditions.  From \cref{thm:ioeq}, the $c_j$ coefficient of the input-output equation of the skeletal path model $\cm_i$ is going to be the sum of the productivities of all incoming forests on $\widetilde{P}_n$ with $n-j$ edges, i.e. 
\[c_j := \sum_{F \in \mathcal{F}_{n-j}(\widetilde{P}_n)} \pi_F.
\]
Note that $\widetilde{P}_n$ is the directed path graph $P_n$ with an additional vertex labeled ``$0$'' and an additional edge $i \to 0$.  

Elementary symmetric polynomials also follow a similar structure pattern, where $\sigma_{n-j}(Q_i)$ is the sum of all possible combinations of $n-j$ parameters (or edges).
Note that $\sigma_{n-j}(Q_i)$ contains the sum of the productivity of every possible incoming forests on $\widetilde{P}_n$ with $n-j$ edges, and possibly other combinations of parameters that do not correspond to incoming forests on $\widetilde{P}_n$.  Symbolically, we can write this relationship as 
\begin{equation}\label{eq:sigma}
\sigma_{n-j}(Q_i) = \left( \sum_{F \in \mathcal{F}_{n-j}(\widetilde{P}_n)} \pi_F \right) + B_{n-j}\end{equation}
where $B_{n-j}$ is the sum of all combinations of $n-j$ edges that do \textbf{not} correspond to incoming forests on $\widetilde{P}_n$.  

Now, the skeletal path model $\cm_i$ is defined with exactly one leak in compartment $i$ with $1\leq i \leq n-1$ and parameters $Q_i = \lbrace a_{21}, \ldots , a_{n(n-1)}, a_{0i} \rbrace $. Note that there does not exist a cycle in the $\widetilde{P}_n$ graph corresponding to the skeletal path model $\cm_i$, because there are no cycles in the underlying undirected graph $P_n$, and the addition of the vertex $0$ and the edge $i\to 0$ does not create a cycle.  Thus, the only possible non-incoming forests in $\sigma_{n-j}(Q_i)$ for the model $\cm_i$ are the products containing two edges which leave the same vertex.  In the path model $\cm_i$, this can only occur in the leak vertex, with the edges $a_{0i}$ and $a_{(i+1)i}$. 

Therefore, we can generate all possible non-incoming forests on $\widetilde{P}_n$ with $n-j$ edges by generating the symmetric polynomial with $n-j-2$ parameters on the set of parameters not including the two edges $a_{0i}$ and $a_{(i+1)i}$, and multiplying $a_{0i}$ and $a_{(i+1)i}$ to them, i.e. 
\[B_{n-j}=(a_{0i}a_{(i+1)i}) \cdot \sigma_{n-j-2}(Q_i\setminus\{a_{0i},a_{(i+1)i}\}).\]  Thus, we can rewrite \cref{eq:sigma} as 
\[
\sigma_{n-j}(Q_i)=\underbrace{\sum_{F \in \mathcal{F}_{n-j}(\widetilde{P}_n)} \pi_F}_{\text{incoming forests}}  + \underbrace{(a_{0i}a_{(i+1)i}) \cdot \sigma_{n-j-2}(Q_i \setminus \{a_{0i},a_{(i+1)i}\})}_{\text{\textbf{non}-incoming forests}}
\]
Manipulating this equation, we get 
\[
c_j := \sum_{F \in \mathcal{F}_{n-j}(\widetilde{P}_n)} \pi_F = \sigma_{n-j}(Q_i)-(a_{0i}a_{(i+1)i}) \cdot \sigma_{n-j-2}(Q_i \setminus \{a_{0i},a_{(i+1)i}\})
\]
as desired.

Note that according to \cref{thm:ioeq}, the coefficient $d_j$ is generated by the incoming forests on $\widetilde{P}_n$ with $n-j-1$ edges and a path from the input $1$ to the output $n$, i.e.
\[
d_j = \sum_{F \in \mathcal{F}_{n-j-1}^{1,n}(\widetilde{P}_{n}^*)}.
\]
Note that $\widetilde{P}_{n} = \widetilde{P}_n^\ast$ since our output node $n$ in $\widetilde{P}_n$ has no edges leaving it.  In the graph $\widetilde{P}_n$, we have exactly one path from the input $1$ to the output $n$, which is every edge along the path $P_n$.  The only parameter not on the path from $1$ to $n$ in $\widetilde{P}_n$ is the leak parameter $a_{0i}$, which cannot be added to any coefficient containing the path from $1$ to $n$, since this path already includes the edge $i \to i+1$ ($a_{(i+1)i}$ parameter), which would break the incoming condition.  Thus the only non-zero coefficient $d_j$ is 
\[d_0 := \sum_{F \in \mathcal{F}_{n-0-1}^{1,n}(\widetilde{P}_{n})} \pi_F = a_{21}a_{32}\cdots a_{n(n-1)}
\]
as desired.

\end{proof}

\begin{ex}\label{ex:firstsigma}
Consider again the model $\cm_3=(P_4,\{1\},\{4\},\{3\})$ from \cref{ex:path} where we define $Q_3:=\mathcal{P}(\cm_3) = \{a_{03},a_{21},a_{32},a_{43}\}$.
Using the previous proposition, we can generate the coefficients of the input-output equation via:
\begin{align*}
c_0 &= \sigma_{4-0}(Q_3) - \left(a_{03}a_{(3+1)3}\right)\cdot  \sigma_{4-0-2}(Q_3\setminus\{a_{03}a_{(3+1)3}\}) \\
&= a_{21}a_{32}a_{43}a_{03} - (a_{03}a_{43})\cdot (a_{21}a_{32})\\
&= 0 \\
c_1 &= \sigma_{4-1}(Q_3) - \left(a_{03}a_{(3+1)3}\right)\cdot  \sigma_{4-1-2}(Q_3\setminus\{a_{03}a_{(3+1)3}\}) \\
&= a_{21}a_{32}a_{43} + a_{21}a_{32}a_{03} + a_{21}a_{43}a_{03} + a_{32}a_{43}a_{03} - (a_{03}a_{43}) \cdot (a_{21}) - (a_{03}a_{43})\cdot (a_{32}) \\
&= a_{21}a_{32}a_{43} + a_{21}a_{32}a_{03}\\
c_2 &= \sigma_{4-2}(Q_3) - \left(a_{03}a_{(3+1)3}\right)\cdot  \sigma_{4-2-2}(Q_3\setminus\{a_{03}a_{(3+1)3}\})   \\
&=a_{21}a_{32}+a_{21}a_{43}+a_{21}a_{03}+a_{32}a_{43}+a_{32}a_{03}+a_{43}a_{03} - (a_{03}a_{43})\cdot 1\\
&=a_{21}a_{32}+a_{21}a_{43}+a_{21}a_{03}+a_{32}a_{43}+a_{32}a_{03}\\
c_3 &= \sigma_{4-3}(Q_3) - \left(a_{03}a_{(3+1)3}\right)\cdot  \sigma_{4-3-2}(Q_3\setminus\{a_{03}a_{(3+1)3}\}) \\
&=a_{21}+a_{32}+a_{43}+a_{03} - (a_{03}a_{43})\cdot 0\\
&=a_{21}+a_{32}+a_{43}+a_{03}\\
d_0 &= a_{21}a_{32}a_{43}\\
\end{align*}
This aligns exactly with our initial derivation of the input-output equation of this model found in both \cref{ex:ioequation} using differential substitution, and \cref{ex:iobygraph} using the general graph theoretic generation.
\end{ex}

Now, we state and prove a similar result related to our model $\cm_n$.

\begin{proposition}\label{prop:secondsigma}
    Consider a linear compartmental model $\mathcal{M}_{n}=(H, \{1\}, \{n\}, \emptyset)$ where $H=P_n \cup \{n\to n-1\}$ with a single input and output.  Write the input-output equation as: 
	\begin{equation}\label{eq:inout}
	y_{n}^{(n)}+c_{n-1}y_{n}^{(n-1)}+\cdots + c_1y_{n}'
	 ~= ~ d_0u_{1}~.
	\end{equation} 
 Let $Q_n = \mathcal{P}(\cm_n)$.  Then the coefficients of this input-output equation are as follows:
\begin{align*}
c_j &= \sigma_{n-j}(Q_n) - \left(a_{n(n-1)}a_{(n-1)n}\right)\sigma_{n-j-2}(Q_n\setminus\{a_{n(n-1)},a_{(n-1)n}\}) \ \text{ for } j \in \{1,2,\ldots, n-1\}\\
d_0 &= a_{21}a_{32}\cdots a_{n,n-1}.
\end{align*}    
\end{proposition}

\begin{proof} As in the proof of the previous proposition, we will again note that 
\begin{equation}\label{eq:sigma2}
\sigma_{n-j}(Q_n) = \left( \sum_{F \in \mathcal{F}_{n-j}(\widetilde{H})} \pi_F \right) + B_{n-j}\end{equation}
where the left summand represents the $c_j$ coefficients, and thus the incoming forests on $\widetilde{H}$ with $n-j$ edges, and the $B_{n-j}$ represents combinations of $n-j$ edges in $\widetilde{H}$ that are \textbf{not} incoming forests.

The skeletal path model $\cm_n$ has parameters $Q_n = \lbrace a_{21}, \ldots , a_{n(n-1)}, a_{(n-1)n} \rbrace $, and contains no leaks.  Again, as in the previous proof, we consider the possible combinations of edges that could break either the incoming or forest conditions of an incoming forest.  In the case of $\widetilde{H}$, each vertex has exactly one outgoing edge, meaning that no combination of edges could break the incoming condition.  This is, however, exactly one cycle in $\widetilde{H}$, namely the cycle between $n-1$ and $n$. Therefore, we can generate all possible combinations of $n-j$ parameters including these two edges by generating the elementary symmetric polynomial with $n-j-2$ parameters on the set of parameters not including these edges, and multiplying $a_{n(n-1)}$ and $a_{(n-1)n}$ to them, i.e. 
\[B_{n-j}=(a_{n(n-1)}a_{(n-1)n}) \cdot \sigma_{n-j-2}(Q_n\setminus\{a_{n(n-1)}a_{(n-1)n}\}).\]  Thus, we can rewrite \cref{eq:sigma2} as 
\[
\sigma_{n-j}(Q_n)=\underbrace{\sum_{F \in \mathcal{F}_{n-j}(\widetilde{H})} \pi_F}_{\text{incoming forests}}  + \underbrace{(a_{n(n-1)}a_{(n-1)n}) \cdot \sigma_{n-j-2}(Q_n \setminus \{a_{n(n-1)}a_{(n-1)n}\})}_{\text{\textbf{not} incoming forests}}
\]
Manipulating this equation, we get 
\[
c_j := \sum_{F \in \mathcal{F}_{n-j}(\widetilde{H})} \pi_F = \sigma_{n-j}(Q_n)-(a_{n(n-1)}a_{(n-1)n}) \cdot \sigma_{n-j-2}(Q_n \setminus \{a_{n(n-1)}a_{(n-1)n}\})
\]
as desired.

For the coefficient of $u_1$, we note again that according to \cref{thm:ioeq}, the coefficient $d_j$ is generated by the incoming forests on $\widetilde{H}^*$ with $n-j-1$ edges and a path from the input 1 to the output $n$, i.e.
\[
d_j := \sum_{F \in \mathcal{F}_{n-j-1}^{1,n}(\widetilde{H}^*)} \pi_F
\]
Note here that $\widetilde{H}^*=P_n$ since we remove any edge which leaves the output vertex, in this case $n\to n-1$, which was the only edge in $H$ that was not on the path from 1 to $n$.  Thus, as in the previous proof, we have exactly one incoming forest (with any number of edges) on $P_n$ that contains a path from $1$ to $n$, namely $P_n$ itself, so
\[d_0 := \sum_{F \in \mathcal{F}_{n-0-1}^{1,n}(\widetilde{H}^*)} \pi_F = \sum_{F \in \mathcal{F}_{n-0-1}^{1,n}(P_n)} \pi_F = a_{21}a_{32}\cdots a_{n(n-1)}
\]
as desired.
\end{proof}

Using these two propositions, we will now prove the main two theorems outlined in \cref{sect:indist}.  First, we prove the indistinguishability of $\cm_{n-1}$ and $\cm_n$.

\begin{thm}\label{thm:indist_cycle}
    The path models $\cm_{n-1}=(P_n,\{1\},\{n\},\{n-1\})$ and $\cm_n = (H,\{1\},\{n\},\emptyset)$ where $H=P_n \cup \{n \to n-1\}$ are permutation indistinguishable.
\end{thm}

\begin{proof} 

Recall that two models $\cm_i$ and $\cm_j$ are permutation indistinguishable if there exists a bijective map $\Phi$ from the set of parameters of $\cm_{n-1}$, i.e.~$Q_{n-1}=\{a_{21},a_{32},\ldots , a_{n(n-1)},a_{0(n-1)}\}$, to the set of parameters of $\cm_n$, called $Q_n=\{b_{21},b_{32},\ldots , b_{n(n-1)},b_{(n-1)n}\}$, such that the coefficients of $\cm_{n-1}$ map exactly to the coefficients of $\cm_n$ under $\Phi$.  For these models, we propose the following set isomorphism:
\[
\Phi \colon \{\underbrace{a_{21},a_{32},\ldots , a_{n(n-1)},a_{0(n-1)}}_{Q_{n-1}}\} \to \{\underbrace{b_{21},b_{32},\ldots , b_{n(n-1)},b_{(n-1)n}}_{Q_n}\}
\]
where
    \[\Phi(a_{uv}) = \begin{cases}
    b_{(n-1)n} & \text{if } (u,v) = (0,n-1) \\ 
    b_{uv} & \text{otherwise}. \\
    \end{cases}
\]
Note that we change the parameters of $\cm_n$ to be $b_{ij}$ to distinguish what model's parameters we are referencing.

We know that the sets, $Q_{n-1}$ and $Q_n$ have the same cardinality, namely $n$.  Also, $\Phi$ is injective since each parameters maps to its own unique parameter in the image.  Therefore, $\Phi$ is an injective map between $Q_{n-1}$ and $Q_n$ which are finite sets with the same cardinality, meaning $\Phi$ is bijective and hence an isomorphism (under the definition that $\Phi(uv) = \Phi(u)\Phi(v)$ and $\Phi(u+v)=\Phi(u)+\Phi(v)$).

By \cref{prop:secondsigma}, the left-hand coefficient of the input-output equations of $\cm_{n-1}$ and $\cm_n$ respectively are
\begin{align*}
    c_j^{[n-1]} &= \sigma_{n-j}(Q_{n-1}) - (a_{0(n-1)}a_{n(n-1)})\sigma_{n-j-2}(Q_{n-1}\setminus\{a_{0(n-1)},a_{n(n-1)}\}) \\
    c_j^{[n]} &= \sigma_{n-j}(Q_n) - (b_{(n-1)n}b_{n(n-1)})\sigma_{n-j-2}(Q_n \setminus\{b_{(n-1)n},b_{n(n-1)}\}) \\
\end{align*}
where the superscripts denote which model the coefficients belong to.

Note that since $\Phi$ is a set isomorphism, it respects the additive and multiplicative structures, thus 
\begin{align*}
\Phi\left(c_j^{[n-1]}\right) &= \Phi \left(\sigma_{n-j}(Q_{n-1}) - (a_{0(n-1)}a_{n(n-1)})\sigma_{n-j-2}(Q_{n-1} \setminus\{a_{0(n-1)},a_{n(n-1)}\}) \right) \\
&= \Phi \left(\sigma_{n-j}(Q_{n-1})\right) - \Phi \left((a_{0(n-1)}a_{n(n-1)}) \right) \Phi \left(\sigma_{n-j-2}(Q_{n-1} \setminus\{a_{0(n-1)},a_{n(n-1)}\}) \right) \\
& = \sigma_{n-j}(Q_n) - (b_{(n-1)n}b_{n(n-1)})\sigma_{n-j-2}(Q_n \setminus\{b_{(n-1)n},b_{n(n-1)}\}) \\
&= c_j^{[n]}
\end{align*}

First, consider $\Phi \left(\sigma_{n-j}(Q_{n-1})\right)$, where $\sigma_{n-j}(Q_{n-1})$ is the sum of all possible combinations of $n-j$ elements of $Q_{n-1}$, where $|Q_{n-1}| = n$.  Note, that since $\sigma_{n-j}(Q_n)$ is the sum of all possible combinations of $n-j$ elements of $Q_n$, where $|Q_n|=n$, and because $\Phi$ is a bijection between $Q_{n-1}$ and $Q_n$, then
\[\Phi \left(\sigma_{n-j}(Q_{n-1})\right)=\sigma_{n-j}\left(Q_n\right).
\]

Second, note that $\Phi \left(a_{0(n-1)}a_{n(n-1)} \right)=\Phi \left(a_{0(n-1)}\right) \Phi\left(a_{n(n-1)} \right)=b_{(n-1)n}b_{n(n-1)}$ by the definition of $\Phi$.

Finally, consider $\Phi \left(\sigma_{n-j-2}(Q_{n-1} \setminus\{a_{0(n-1)},a_{n(n-1)}\})\right)$.  Note that $\Phi$ maps $Q_{n-1} \setminus\{a_{0(n-1)},a_{n(n-1)}\}$ bijectively to $Q_n \setminus\{b_{(n-1)n},b_{n(n-1)}\}$, meaning that $\Phi$ will map an elementary symmetric polynomial with $n-j-2$ elements of $Q_{n-1} \setminus\{a_{0(n-1)},a_{n(n-1)}\}$ to the elementary symmetric polynomial with $n-j-2$ elements of $Q_n \setminus\{b_{(n-1)n},b_{n(n-1)}\}$, i.e.

\[\Phi \left(\sigma_{n-j-2}(Q_{n-1} \setminus\{a_{0(n-1)},a_{n(n-1)}\})\right) = \sigma_{n-j-2}(Q_n \setminus\{b_{(n-1)n},b_{n(n-1)}\})
\]

For the right-hand coefficient, note that
\begin{align*}
    d_0^{[n-1]} &= a_{21}a_{32}\cdots a_{n(n-1)} \\
    d_0^{[n]} &= b_{21}b_{32}\cdots b_{n(n-1)}. \\
\end{align*}
Therefore, by the definition of $\Phi$, 
\begin{align*}
    \Phi(d_0^{[n-1]}) &= \Phi(a_{21}a_{32} \cdots  a_{n(n-1)}) \\ 
    &= \Phi(a_{21})\Phi(a_{32})\cdots \Phi(a_{n(n-1)}) \\ 
    &= b_{21}b_{32}\cdots  b_{n(n-1)} \\
    &= d_0^{[n]}.
\end{align*}
Thus, the models $\cm_{n-1}$ and $\cm_n$ are permutation indistinguishable under the map $\Phi$ as defined.
\end{proof}

Now we move onto the proof of the other main theorem, namely that any skeletal path model $\cm_i$ is indistinguishable with another path model $\cm_k$ as long as $1\leq i,k<n$.  This proof is very similar to the previous proof with a slightly more complicated indistinguishability map.

\begin{thm} 
    The path models $\cm_i=(P_n,\{1\},\{n\},\{i\})$ and $\cm_k = (P_n,\{1\},\{n\},\{k\})$ are [permutation] indistinguishable for all $1\leq i <k <n$.
\end{thm}

\begin{proof} 
For these models $\cm_i$ and $\cm_k$, we propose the following set isomorphism:
\[
\Phi \colon \{\underbrace{a_{21},a_{32},\ldots , a_{n(n-1)},a_{0i}}_{Q_i}\} \to \{\underbrace{b_{21},b_{32},\ldots , b_{n(n-1)},b_{0k}}_{Q_k}\}
\]
where
    \[\Phi(a_{uv}) = \begin{cases}
    b_{0k} & \text{if } (u,v) = (0,i) \\
    b_{k+1,k} & \text{if } (u,v) = (i+1,i) \\
    b_{i+1,i} & \text{if } (u,v) = (k+1,k) \\
    b_{uv} & \text{otherwise}. \\
    \end{cases}
    \]
In other words, we map the leak to the leak, and swap the edges on the paths of both models occurring out of the leak vertex.

Again, we know that the sets, $Q_i$ and $Q_k$ have the same cardinality and that $\Phi$ is injective since each element maps to its own unique element in the image, meaning $\Phi$ is an isomorphism.

By \cref{prop:firstsigma}, the left-hand coefficient of the input-output equations of $\cm_i$ and $\cm_k$ respectively are
\begin{align*}
    c_j^{[i]} &= \sigma_{n-j}(Q_i) - (a_{0i}a_{(i+1)i})\sigma_{n-j-2}(Q_i \setminus\{a_{0i},a_{(i+1)i}\}) \\
    c_j^{[k]} &= \sigma_{n-j}(Q_k) - (b_{0k}b_{(k+1)k})\sigma_{n-j-2}(Q_k \setminus\{b_{0k},b_{(k+1)k}\}) \\
\end{align*}
where the superscripts denote which model the coefficients belong to.

Note that since $\Phi$ is a set isomorphism, it respects the additive and multiplicative structures, thus 
\begin{align*}
\Phi\left(c_j^{[i]}\right) &= \Phi \left(\sigma_{n-j}(Q_i) - (a_{0i}a_{(i+1)i})\sigma_{n-j-2}(Q_i \setminus\{a_{0i},a_{(i+1)i}\}) \right) \\
&= \Phi \left(\sigma_{n-j}(Q_i)\right) - \Phi \left((a_{0i}a_{(i+1)i}) \right) \Phi \left(\sigma_{n-j-2}(Q_i \setminus\{a_{0i},a_{(i+1)i}\}) \right) \\
& = \sigma_{n-j}(Q_k) - (b_{0k}b_{(k+1)k}) \sigma_{n-j-2}(Q_k \setminus\{b_{0k},b_{(k+1)k}\}) \\
&= c_j^{[k]}
\end{align*}
First, again because $\Phi$ is a bijection between $Q_i$ and $Q_k$, it maps the elementary symmetric polynomials with a certain number of elements on $Q_i$ to the corresponding elementary symmetric polynomials with that number of elements on $Q_k$:
\[\Phi \left(\sigma_{n-j}(Q_i)\right)=\sigma_{n-j}\left(Q_k\right).
\]
Second, note that $\Phi \left(a_{0i}a_{(i+1)i} \right)=\Phi \left(a_{0i}\right) \Phi\left(a_{(i+1)i} \right)=b_{0k}b_{(k+1)k}$ by the definition of $\Phi$.

Finally, consider $\Phi \left(\sigma_{n-j-2}(Q_i \setminus\{a_{0i},a_{(i+1)i}\})\right)$.  Note that $\Phi$ maps $Q_i \setminus\{a_{0i},a_{(i+1)i}\}$ bijectively to $Q_k \setminus\{b_{0k},b_{(k+1)k}\}$, meaning again that

\[\Phi \left(\sigma_{n-j-2}(Q_i \setminus\{a_{0i},a_{(i+1)i}\})\right) = \sigma_{n-j-2}(Q_k \setminus\{b_{0k},b_{(k+1)k}\})
\]

For the coefficient of $u_1$, \cref{prop:firstsigma} tells us that
\begin{align*}
    d_0^{[i]} &= a_{21}a_{32}\cdots a_{n(n-1)} \\
    d_0^{[k]} &= b_{21}b_{32}\cdots b_{n(n-1)}. \\
\end{align*}
Therefore, by the definition of $\Phi$, 
\begin{align*}
    \Phi(d_0^{[i]}) &= \Phi(a_{21}a_{32}\cdots a_{(i+1)i} \cdots a_{(k+1)k}\cdots  a_{n(n-1)}) \\
    &= \Phi(a_{21})\Phi(a_{32})\cdots \Phi(a_{(i+1)i}) \cdots \Phi(a_{(k+1)k})\cdots \Phi(a_{n(n-1)}) \\ 
    &= b_{21}b_{32}\cdots b_{(k+1)k} \cdots b_{(i+1)i} \cdots  b_{n(n-1)} \\
    &= d_0^{[k]}.
\end{align*}

Therefore, the models $\cm_{i}$ and $\cm_{k}$ are permutation indistinguishable under the map $\Phi$ as defined.
\end{proof}

\section{Conclusion}
In this work, we have reproved results of Bortner and Meshkat related to the sufficient conditions for indistinguishability of certain subclasses of skeletal path models.  Specifically, we focused on the indistinguishability of a model consisting of an underlying path, and any leak along the path, or the edge from the last vertex to the second to last vertex in the path using a graph theoretic approach.  We believe this graph theoretic approach could be extended to prove indistinguishability of other families of models that have well understood graphical structure, but less well understood compartmental matrix structure.

\bibliographystyle{plain}
\bibliography{indist-ug}

\begin{thebibliography}{10}

\bibitem{blackwood2018introduction}
Julie~C Blackwood and Lauren~M Childs.
\newblock An introduction to compartmental modeling for the budding infectious
  disease modeler.
\newblock {\em Lett.\ Biomath.}, 5(1):195--221, 2018.

\bibitem{aim}
Cashous Bortner, Elizabeth Gross, Nicolette Meshkat, Anne Shiu, and Seth
  Sullivant.
\newblock Identifiability of linear compartmental tree models and a general
  formula for the input-output equations.
\newblock {\em Advances in Applied Mathematics}, 146, May 2023.

\bibitem{bortner-meshkat}
Cashous Bortner and Nicolette Meshkat.
\newblock Identifiable paths and cycles in linear compartmental models.
\newblock {\em Bulletin of mathematical biology}, 84(5):53, March 2022.

\bibitem{BortnerMeshkat-Indist}
Cashous Bortner and Nicolette Meshkat.
\newblock Graph-based sufficient conditions for the indistinguishability of
  linear compartmental models.
\newblock {\em SIAM Journal on Applied Dynamical Systems}, 23(3):2179--2207,
  2024.

\bibitem{CJSSS}
Patrick Chan, Katherine Johnston, Anne Shiu, Aleksandra Sobieska, and Clare
  Spinner.
\newblock Identifiability of linear compartmental models: {T}he impact of
  removing leaks and edges.
\newblock {\em Available from {\tt arXiv:2102.04417}}, 2021.

\bibitem{dipiro2010concepts}
Joseph~T DiPiro.
\newblock {\em Concepts in clinical pharmacokinetics}.
\newblock ASHP, 2010.

\bibitem{distefano-book}
J.~J. DiStefano, III.
\newblock {\em Dynamic systems biology modeling and simulation}.
\newblock Academic Press, 2015.

\bibitem{godfrey}
Keith Godfrey.
\newblock {\em Compartmental Models and their Application}.
\newblock Academic Press, 1983.

\bibitem{godfrey-chapman}
Keith~R. Godfrey and Michael~J. Chapman.
\newblock Identifiability and indistinguishability of linear compartmental
  models.
\newblock {\em Mathematics and Computers in Simulation}, 32:273--295, 1990.

\bibitem{singularlocus}
Elizabeth Gross, Nicolette Meshkat, and Anne Shiu.
\newblock Identifiability of linear compartmental models: The singular locus.
\newblock {\em Advances in Applied Mathematics}, 133:102268, 2022.

\bibitem{gydesen1984mathematical}
Helge Gydesen.
\newblock Mathematical models of the transport of pollutants in ecosystems.
\newblock {\em Ecol.\ Bull.}, pages 17--25, 1984.

\bibitem{hedaya2012basic}
Mohsen~A Hedaya.
\newblock {\em Basic pharmacokinetics}.
\newblock CRC Press, 2012.

\bibitem{meshkat-rosen-sullivant}
Nicolette Meshkat, Zvi Rosen, and Seth Sullivant.
\newblock Algebraic tools for the analysis of state space models.
\newblock In {\em The 50th anniversary of {G}r\"{o}bner bases}, volume~77 of
  {\em Adv. Stud. Pure Math.}, pages 171--205. Math.\ Soc.\ Japan, Tokyo, 2018.

\bibitem{MeshkatSullivantEisenberg}
Nicolette Meshkat, Seth Sullivant, and Marisa Eisenberg.
\newblock Identifiability results for several classes of linear compartment
  models.
\newblock {\em Bull.\ Math.\ Biol.}, 77(8):1620--1651, 2015.

\bibitem{Ovchinnikov-Pogudin-Thompson}
Alexey Ovchinnikov, Gleb Pogudin, and Peter Thompson.
\newblock Input-output equations and identifiability of linear ode models.
\newblock {\em IEEE Transactions on Automatic Control}, 68(2):812--824, 2023.

\bibitem{tozer1981concepts}
Thomas~N Tozer.
\newblock Concepts basic to pharmacokinetics.
\newblock {\em Pharmacol.\ Toxicol.}, 12(1):109--131, 1981.

\bibitem{wagner1981history}
John~G Wagner.
\newblock History of pharmacokinetics.
\newblock {\em Pharmacol.\ Toxicol.}, 12(3):537--562, 1981.

\bibitem{zhang1991}
Li-Qun Zhang, Jerry~C. Collins, and Paul~H. King.
\newblock Indistinguishability and identifiability analysis of linear
  compartmental models.
\newblock {\em Mathematical Biosciences}, 103(1):77--95, 1991.

\end{thebibliography}

\end{document}